\documentclass[a4paper,11pt]{article}
\usepackage{}
\usepackage{mathrsfs}
\usepackage{amsfonts}
\usepackage{amsfonts}
\usepackage{amssymb}
\usepackage{mathrsfs}
\usepackage{indentfirst,latexsym,bm,amsmath,amssymb,amsthm}
\usepackage{xcolor}
\usepackage[
                   bookmarksnumbered=true,
                   bookmarksopen=true, 
                   pdfauthor=kmc,
                   pdfcreator={LaTex with hyperref package + WinEdt},
                   pdftitle=trudge,
                   colorlinks,%
                   citecolor=blue,
                   linkcolor=blue,%
                   urlcolor=blue,
                   hyperindex,%
                   plainpages=false,%
                   pdfstartview=FitH,
                   linktocpage=true,
                   dvipdfm%
                   ]{hyperref}
\usepackage{indentfirst,latexsym,bm,amsmath,amssymb,amsthm}
\usepackage[dvips]{graphicx}

\makeatletter\@addtoreset{equation}{section} \makeatother
\newtheorem{thm}{Theorem}[section]
\newtheorem{Lemma}{Lemma}[section]

\newtheorem{rem}{Remark}[section]

 

\title{On one-dimension quasilinear wave equations with\\
null conditions}

\author{ Dongbing Zha\thanks{Department of Mathematics and Institute for Nonlinear Sciences, Donghua University, Shanghai 201620, PR China.{ E-mail address: ZhaDongbing@163.com}}}

\begin{document}

\maketitle
\begin{abstract}
In this paper, we show that one-dimension systems of quasilinear wave equations with null conditions admit global classical solutions for small initial data. This result extends Luli, Yang and Yu's seminal work [G.~K. Luli, S.~Yang, P.~Yu, \href{https://doi.org/10.1016/j.aim.2018.02.022}{On
  one-dimension semi-linear wave equations with null conditions}, Adv. Math.
  329 (2018) 174--188] from the semilinear case to the quasilinear case. Furthermore, we also prove that the global solution is asymptotically free in the energy sense. In order to achieve these goals, we will employ Luli, Yang and Yu's weighted energy estimates with positive weights, introduce some space-time weighted energy estimates and pay some special attentions to the highest order energies,
 then use some suitable bootstrap process to close the argument. \\
\emph{keywords}: One-dimension quasilinear wave equations; null conditions; global existence.\\
\emph{2010 MSC}: 35L05; 35L72.
\end{abstract}

\pagestyle{plain} \pagenumbering{arabic}

\section{Introduction and main result}
In \cite{MR3783412}, for Cauchy problems of semilinear wave equations with null conditions in one space dimension, Luli, Yang and Yu proved the global existence of classical solutions with small initial data (a former result can be found in \cite{MR3121700}). This result can be viewed as a one-dimension and semilinear analogue of the pioneering works Klainerman \cite{Klainerman86} and Christodoulou \cite{MR820070} for the global existence of classical solutions for nonlinear wave equations with null conditions in three space dimensions, and of Alinhac \cite{Alinhac01} for the case of two space dimensions.\par
Global existence of small solutions to nonlinear wave equations with null conditions
has been a subject under active investigation for the past four decades. The approach to understand the small data problem is based on the decay mechanism of linear waves.
It is well-known that the decay rate is $(1+t)^{-\frac{d-1}{2}}$, for $d$-dimension linear waves. Thus, when $d\geq 4$, small-data-global-existence type theorems hold for generic quadratic nonlinearities, since the decay rate is integrable in time. See Klainerman \cite{MR544044,Klainerman85}. However, in $\mathbb{R}^{1+3}$, the slower decay rate $(1+t)^{-1}$ just barely fails to be integrable in time. Hence, the solution may blowup at finite time. See John \cite{John1}. In order to avoid the formation of singularities,
Klainerman \cite{Klainerman82} introduced the celebrated null conditions, then for three dimensional quasilinear wave equations with quadratic nonlinearities satisfying null conditions, based on the decay rate for linear waves and some special cancelations provided by such structural conditions,
Klainerman \cite{Klainerman86} and Christodoulou \cite{MR820070} independently proved the global existence theorems. As for $\mathbb{R}^{1+2}$, linear waves only admit decay rate $(1+t)^{-1/2}$, which is far from integrable in time. Nevertheless, for a class of two dimensional quasilinear wave equations, based on such decay rate for linear waves and the null conditions imposed on quadratic and cubic nonlinearities, Alinhac \cite{Alinhac01} can show the global existence theorem. For the general class of two dimensional quasilinear wave equations satisfying null conditions,  we refer the reader to Katayama
\cite{MR3729247} and Zha \cite{MR3912654}.
For more detailed explanations on the concept of null conditions, we refer the reader to Luli, Yang and Yu \cite{MR3783412}.
\par
As mentioned before,
 the proofs in \cite{Alinhac01}, \cite{MR820070}, \cite{MR544044}, \cite{Klainerman85} and \cite{Klainerman86} in high space dimension case are all based on the decay mechanism of linear waves. However, in one space dimension case waves do not decay, and any nonlinear resonance (even arbitrarily
high order) can lead to finite time blowup. Nevertheless, for one-dimension semilinear wave equations,
Luli, Yang and Yu \cite{MR3783412} can prove that small data still lead to global solutions if the null condition is satisfied.
Different from the high space dimension case,
 the mechanism for the global existence in one-dimension case
is the interaction of waves with different speeds, which will lead to the decay of nonlinear terms. In order to display this mechanism,  Luli, Yang and Yu developed a kind of weighted energy estimate with positive weights.\par
In this paper, we will consider the case for quasilinear wave equations, which arise naturally in many physical fields.
For this purpose, instead of the standard Cartesian coordinates $(t,x)$, we will mainly use the null coordinates
\begin{align}
\xi=\frac{t+x}{2},~~\eta=\frac{t-x}{2}.
\end{align}
We have the null vector fields
\begin{align}
\partial_{\xi}=\partial_t+\partial_x,~~\partial_{\eta}=\partial_t-\partial_x,
\end{align}
and also denote briefly $u_{\xi}=\partial_{\xi}u$ and $u_{\eta}=\partial_{\eta}u$.\par
It is well known that the one-dimension linear wave equation in the null coordinates $(\xi,\eta)$ can be written as $u_{\xi\eta}=0$. We will treat some quasilinear perturbation of it.
Consider the following one-dimension system of quasilinear wave equations
\begin{align}\label{quasiwave}
u_{\xi\eta}&=A_1(u,u_{\xi},u_{\eta})u_{\xi\eta}
+A_2(u,u_{\xi},u_{\eta})u_{\xi\xi}+A_3(u,u_{\xi},u_{\eta})u_{\eta\eta}+F(u,u_{\xi},u_{\eta}),
\end{align}
where the unknown function $u=u(t,x): \mathbb{R}^{1+1}\longrightarrow \mathbb{R}^{n}$, for $i=1,2,3$, $A_i: \mathbb{R}^{n}\times \mathbb{R}^{n}\times \mathbb{R}^{n}\longrightarrow  \mathbb{R}^{n\times n}$ are given smooth and matrix valued functions, and $F: \mathbb{R}^{n}\times \mathbb{R}^{n}\times \mathbb{R}^{n}\longrightarrow  \mathbb{R}^{n}$ is a given smooth and vector valued function. Moreover, we will always assume that $A_i ~(i=1,2,3)$ are symmetric.
\par
We call that the system \eqref{quasiwave} satisfies the null condition, if near the origin in $\mathbb{R}^{n}\times \mathbb{R}^{n}\times \mathbb{R}^{n}$, it holds that
\begin{align}\label{order1}
A_1(u,u_{\xi},u_{\eta})&=\mathscr{O}(|u|+|u_{\xi}|+|u_{\eta}|),\\\label{order22}
A_2(u,u_{\xi},u_{\eta})&=\mathscr{O}(|u_{\eta}|),\\\label{order33}
A_3(u,u_{\xi},u_{\eta})&=\mathscr{O}(|u_{\xi}|),\\\label{order4}
F(u,u_{\xi},u_{\eta})&=\mathscr{O}(|u_{\xi}||u_{\eta}|).
\end{align}
Inspired by Luli, Yang and Yu's seminal result \cite{MR3783412} in the semilinear case, it is natural to conjecture that the Cauchy problem of one-dimension system of quasilinear wave equations \eqref{quasiwave} satisfying null conditions \eqref{order1}--\eqref{order4} admits a unique global classical solution for small initial data. The main aim of this paper is to verify this conjecture. Furthermore, we also want to clarify the asymptotic behavior of the global solution.
\par
Now we introduce some vector fields and energies used in the following part of this paper.
We will use
 \begin{align}
 Z=(\partial_{\xi},\partial_{\eta})
  \end{align}
 as the commuting vector fields. For a multi-index $a=(a_1,a_2)$, set
  \begin{align}
  Z^{a}=\partial_{\xi}^{a_1}\partial_{\eta}^{a_2}
  \end{align}
   and $|a|=a_1+a_2$.\par
  Following Luli, Yang and Yu \cite{MR3783412},
we will use the following weighted energy with positive weights
\begin{align}\label{based}
E_1(u(t))=\|\langle \xi\rangle^{1+\delta}u_{\xi}\|_{L_{x}^2(\mathbb{R})}^2+\|\langle \eta\rangle^{1+\delta}u_{\eta}\|_{L_{x}^2(\mathbb{R})}^2,
\end{align}
where $\langle \cdot\rangle=(1+|\cdot|^2)^{1/2}$. Then we use
\begin{align}
E_2(u(t))=\sum_{|a|= 1}E_1(Z^{a}u(t))
\end{align}
to denote the second order energies.
As for the third order (highest order) energies, we have to distinguish some \lq\lq mixed derivatives" from others. Specifically speaking, we will employ the following third order weighted energies
\begin{align}\label{Pe33333}
E_3(u(t))=\sum_{\substack{|a|=2\\a_1\neq 0}}\|\langle \xi\rangle^{1+\delta}Z^{a}u_{\xi}\|_{L_{x}^2(\mathbb{R})}^2+\sum_{\substack{|a|=2\\a_2\neq 0}}\|\langle \eta\rangle^{1+\delta}Z^{a}u_{\eta}\|_{L_{x}^2(\mathbb{R})}^2
\end{align}
and
\begin{align}\label{Pe3fff883333}
\widetilde{E_3}(u(t))&=\sum_{\substack{|a|=2\\a_1=0}}\|\langle \xi\rangle^{1+\delta}Z^{a}u_{\xi}\|_{L_{x}^2(\mathbb{R})}^2+\sum_{\substack{|a|=2\\a_2= 0}}\|\langle \eta\rangle^{1+\delta}Z^{a}u_{\eta}\|_{L_{x}^2(\mathbb{R})}^2\nonumber\\
&=\|\langle \xi\rangle^{1+\delta}u_{\eta\eta\xi}\|_{L_{x}^2(\mathbb{R})}^2+\|\langle \eta\rangle^{1+\delta}u_{\xi\xi\eta}\|_{L_{x}^2(\mathbb{R})}^2.
\end{align}
We also set
\begin{align}\label{ertyuhhhhhh}
E(u(t))=E_1(u(t))+E_2(u(t))+E_3(u(t)).
\end{align}
\par
Inspired by Alinhac \cite{Alinhac01} and Lindblad and Rodnianski \cite{MR2680391}, based on \eqref{based}, we further introduce the following space-time weighted energy
\begin{align}\label{soapj8jnjjk}
\mathcal{E}_1(u(t))=\int_0^{t}\|\langle \eta\rangle^{-\frac{1+\delta}{2}}\langle \xi\rangle^{1+\delta}u_{\xi}\|_{L_{x}^2(\mathbb{R})}^2ds+\int_0^{t}\|\langle \xi\rangle^{-\frac{1+\delta}{2}}\langle \eta\rangle^{1+\delta}u_{\eta}\|_{L_{x}^2(\mathbb{R})}^2ds.
\end{align}
Denote the second order space-time weighted energies by
\begin{align}
\mathcal{E}_2(u(t))=\sum_{|a|=1}\mathcal{E}_1(Z^{a}u(t)).
\end{align}
Corresponding to \eqref{Pe33333} and \eqref{Pe3fff883333}, we will use the following third order space-time weighted energies
\begin{align}
\mathcal{E}_3(u(t))&=\sum_{\substack{|a|=2\\a_1\neq 0}}\int_0^{t}\|\langle \eta\rangle^{-\frac{1+\delta}{2}}\langle \xi\rangle^{1+\delta}Z^au_{\xi}\|_{L_{x}^2(\mathbb{R})}^2ds\nonumber\\
&+\sum_{\substack{|a|=2\\a_2\neq 0}}\int_0^{t}\|\langle \xi\rangle^{-\frac{1+\delta}{2}}\langle \eta\rangle^{1+\delta}Z^au_{\eta}\|_{L_{x}^2(\mathbb{R})}^2ds
\end{align}
and
\begin{align}\label{xuyaoyu78tt8999}
\widetilde{{\mathcal{E}}_3}(u(t))&=\sum_{\substack{|a|=2\\a_1= 0}}\int_0^{t}\|\langle \eta\rangle^{-\frac{1+\delta}{2}}\langle \xi\rangle^{1+\delta}Z^au_{\xi}\|_{L_{x}^2(\mathbb{R})}^2ds+\sum_{\substack{|a|=2\\a_2= 0}}\int_0^{t}\|\langle \xi\rangle^{-\frac{1+\delta}{2}}\langle \eta\rangle^{1+\delta}Z^au_{\eta}\|_{L_{x}^2(\mathbb{R})}^2ds\nonumber\\
&=\int_0^{t}\|\langle \eta\rangle^{-\frac{1+\delta}{2}}\langle \xi\rangle^{1+\delta}u_{\eta\eta\xi}\|_{L_{x}^2(\mathbb{R})}^2ds+\int_0^{t}\|\langle \xi\rangle^{-\frac{1+\delta}{2}}\langle \eta\rangle^{1+\delta}u_{\xi\xi\eta}\|_{L_{x}^2(\mathbb{R})}^2ds.
\end{align}
We also use the notation
\begin{align}\label{dhk8df88976}
\mathcal{E}(u(t))=\mathcal{E}_1(u(t))+\mathcal{E}_2(u(t))+\mathcal{E}_3(u(t)).
\end{align}
\par
We point out that for the system \eqref{quasiwave}, in the special case of $A_2=A_3=0$, it becomes
\begin{align}\label{quasiwave22}
u_{\xi\eta}&=A_1(u,u_{\xi},u_{\eta})u_{\xi\eta}+F(u,u_{\xi},u_{\eta}),
\end{align}
then we can rewrite \eqref{quasiwave22} as the following semilinear system
\begin{align}\label{quasiwave22ff}
u_{\xi\eta}&=\widetilde{F}(u,u_{\xi},u_{\eta})
\end{align}
with $\widetilde{F}=(I-A_1)^{-1}F$ satisfying
\begin{align}\label{order444}
\widetilde{F}(u,u_{\xi},u_{\eta})&=\mathscr{O}(|u_{\xi}||u_{\eta}|).
\end{align}
Thus, noting the equivalent form \eqref{quasiwave22ff}, we can easily get the global existence for the Cauchy problem of \eqref{quasiwave22} by using the result in \cite{MR3783412} (or the proof can be carried out in line with \cite{MR3783412} using \eqref{quasiwave22} directly). In order to treat the general case of $A_2\neq 0$ or $A_3\neq 0$, we must introduce some new approaches. This is the reason why we use the space-time weighted energy \eqref{soapj8jnjjk} on the strip $[0,t]\times \mathbb{R}$, instead of the weighted energy on the characteristical lines in \cite{MR3783412}. This is the first essential difference between the semilinear case and the quasilinear case.
\par
The reason why the third order energies $\widetilde{E_3}(u(t))$ and $\widetilde{{\mathcal{E}}_3}(u(t))$ concerning some \lq\lq mixed derivatives" need to be considered separately is that in our weighted energy estimates, they are not compatible with the null structure of the quasilinear part. In other words, if we treat them by weighted energy estimates, after integrating by parts, some terms in the quasilinear part will be uncontrollable. Thus, they can not be estimated via weighted energy estimates. This fact is also the second essential difference between the semilinear case and the quasilinear case. Fortunately, by using the system \eqref{quasiwave} directly, we find that $\widetilde{E_3}(u(t))$ and $\widetilde{{\mathcal{E}}_3}(u(t))$ can be controlled by $E(u(t))$ and $\mathcal{E}(u(t))$, respectively (see Lemma \ref{keyn9078p}). This observation is a key point in our treatment for the quasilinear part of the system.\par
In order to describe the asymptotic behavior of the global solution, now we introduce some concepts on it.
We say that a function $u=u(t,x)\in C(\mathbb{R}^{+};\dot{H}^1(\mathbb{R}))\cap C^1(\mathbb{R}^{+};L^2(\mathbb{R}))$ is asymptotically free in the energy sense, if there is
$(v_0,v_1)\in \dot{H}^1(\mathbb{R})\times L^2(\mathbb{R})$ such that
\begin{align}\label{profile0}
\lim_{t\rightarrow +\infty}\big(\|u_{\xi}-v_{\xi}\|^2_{L^2_{x}(\mathbb{R})}+\|u_{\eta}-v_{\eta}\|^2_{L^2_{x}(\mathbb{R})}\big)=0,
\end{align}
where $v\in C(\mathbb{R}^{+};\dot{H}^1(\mathbb{R}))\cap C^1(\mathbb{R}^{+};L^2(\mathbb{R}))$ is the unique global solution to the Cauchy problem of homogeneous linear wave equations
\begin{align}\label{profile1}
v_{\xi\eta}=0
\end{align}
with initial data
\begin{align}\label{profile2}
t=0:~v=v_0,~~v_t=v_1.
\end{align}
\par
Consider the Cauchy problem for the system \eqref{quasiwave} with initial data
\begin{align}\label{hj89i}
t=0:~u=u_0,~~u_t=u_1.
\end{align}
Our main result is the following
\begin{thm}\label{mainthm}
Assume that the system \eqref{quasiwave} satisfies the null condition.
Then for all $0<\delta<1$, there exists a positive constant $\varepsilon_0$ such that for any $0<\varepsilon\leq \varepsilon_0$, if
\begin{align}\label{xxxjkdddd900}
\sum_{l=0}^{3}\|\langle x\rangle^{1+\delta}\partial_x^{l}u_0\|_{L_{x}^2(\mathbb{R})}+\sum_{l=0}^{2}\|\langle x\rangle^{1+\delta}\partial_x^{l}u_1\|_{L_{x}^2(\mathbb{R})}\leq \varepsilon,
\end{align}
then the Cauchy problem \eqref{quasiwave}--\eqref{hj89i} admits a unique global classical solution $u$. Moreover, the global solution $u$ is asymptotically free in the energy sense.
\end{thm}
\begin{rem}
For general one-dimension quasilinear wave equations with quadratic nonlinearity
\begin{align}
u_{\xi\eta}=Q(u,u_{\xi},u_{\eta},u_{\xi\xi},u_{\xi\eta},u_{\eta\eta}),
\end{align}
 Li, Yu and Zhou {\rm{\cite{MR1190814}}} showed the classical solution only admits some lifespan of order $\varepsilon^{-1/2}$. Note that the null condition can enhance the lifespan from such
 short time order to infinity.
\end{rem}
\begin{rem}
A small-data-global-existence type result for one-dimension quasilinear hyperbolic systems of diagonal form using the energy method is established in {\rm{\cite{MR4001049}}}.
Some small-data-global-existence type results using the characteristic approach can be found in {\rm{\cite{MR2309570}}}, {\rm{\cite{MR1284811}}}, {\rm{\cite{MR3730736}}}, and {\rm{\cite{MR2544109}}}, etc. See also  {\rm{\cite{MR2045426}}}. For some large-data-global-existence results, we refer the reader to {\rm{\cite{MR596432}}}.
\end{rem}

\begin{rem}
We note that for the system \eqref{quasiwave}, if the null condition is satisfied, any left traveling wave $f(\xi)$ and right traveling wave $g(\eta)$ are solutions to it. An interesting problem is the stability of these traveling wave solutions. Some related results can be found in {\rm{\cite{Wong22}}}, {\rm{\cite{MR3223827}}} and {\rm{\cite{MR3057301}}}.
\end{rem}

The outline of this paper is as follows. In Section \ref{hjsedddd}, some necessary tools used to prove Theorem \ref{mainthm} are introduced.
Section \ref{xhzuyo898} is devoted to the proof of Theorem \ref{mainthm}.
\section{Preliminaries}\label{hjsedddd}
First, by the fundamental theorem of calculus, chain rule and Leibniz's rule, it is easy to get the following two lemmas.
 \begin{Lemma}\label{xuyaouioo999}
 Assume that $A_1=A_1(u,u_{\xi},u_{\eta}), A_2=A_2(u,u_{\xi},u_{\eta})$, $A_3=A_3(u,u_{\xi},u_{\eta})$ and $F=F(u,u_{\xi},u_{\eta})$ satisfy \eqref{order1}, \eqref{order22},
\eqref{order33} and \eqref{order4}, respectively, and
\begin{align}
|u|+|u_{\xi}|+|u_{\eta}|\leq \nu_{0}.
\end{align}
Then we have
 \begin{align}
 |\partial_{\xi}A_1|&\leq C(|u_{\xi}|+|u_{\xi\xi}|+|u_{\xi\eta}|),~~~~~~~~~ |\partial_{\eta}A_1|\leq C(|u_{\eta}|+|u_{\xi\eta}|+|u_{\eta\eta}|),\\
 |\partial_{\xi}A_2|&\leq C|u_{\xi\eta}|+C|u_{\eta}|(|u_{\xi}|+|u_{\xi\xi}|),~~ |\partial_{\eta}A_2|\leq C(|u_{\eta}|+|u_{\xi\eta}|+|u_{\eta\eta}|),\\
 |\partial_{\xi}A_3|&\leq C(|u_{\xi}|+|u_{\xi\xi}|+|u_{\xi\eta}|),~~~~~~~~~|\partial_{\eta}A_3|\leq C|u_{\xi\eta}|+C|u_{\xi}|(|u_{\eta}|+|u_{\eta\eta}|),\\
 |\partial_{\xi}F| &\leq C(|u_{\xi}||u_{\xi\eta}|+|u_{\eta}||u_{\xi\xi}|+|u_{\xi}||u_{\eta}|),\\
 |\partial_{\eta}F| &\leq C(|u_{\eta}||u_{\xi\eta}|+|u_{\xi}||u_{\eta\eta}|+|u_{\xi}||u_{\eta}|),
 \end{align}
 where $C=C(\nu_0)$ is a constant depending on $\nu_0$.
 \end{Lemma}
 \begin{Lemma}\label{xuyaouioo99999999}
 Assume that $A_1=A_1(u,u_{\xi},u_{\eta}), A_2=A_2(u,u_{\xi},u_{\eta})$, $A_3=A_3(u,u_{\xi},u_{\eta})$ and $F=F(u,u_{\xi},u_{\eta})$ satisfy \eqref{order1}, \eqref{order22},
\eqref{order33} and \eqref{order4}, respectively, and
\begin{align}
|u|+|u_{\xi}|+|u_{\eta}|+|Zu_{\xi}|+|Zu_{\eta}|\leq \nu_{1}.
\end{align}
Then for any multi-index $a, |a|=1$, it holds that
\begin{align}
 |Z^{a}A_1|&\leq C(|u_{\xi}|+|u_{\eta}|+|Z^au_{\xi}|+|Z^au_{\eta}|),\\
  |Z^{a}A_2|&\leq C(|u_{\eta}|+|Z^au_{\eta}|),\\
   |Z^{a}A_3|&\leq C(|u_{\xi}|+|Z^au_{\xi}|),\\
   |Z^{a}F|&\leq C(|u_{\xi}||u_{\eta}|+|Z^{a}u_{\xi}||u_{\eta}|+|Z^{a}u_{\eta}||u_{\xi}|).
 \end{align}
And for any multi-index $a,|a|= 2$,
 it holds that
 \begin{align}
 |Z^{a}A_1|&\leq C\sum_{|b|\leq 2}\big(|Z^{b}u_{\xi}|+|Z^{b}u_{\eta}|\big),\\
  |Z^{a}A_2|&\leq C\sum_{|b|\leq 2}|Z^{b}u_{\eta}|+C|u_{\eta}|\sum_{|b|\leq 2}|Z^{b}u_{\xi}|,\\
  |Z^{a}A_3|&\leq C\sum_{|b|\leq 2}|Z^{b}u_{\xi}|+C|u_{\xi}|\sum_{|b|\leq 2}|Z^{b}u_{\eta}|,\\
   |Z^{a}F|&\leq C\sum_{|b|+|c|\leq 2}|Z^{b}u_{\xi}||Z^{c}u_{\eta}|.
 \end{align}
Here $C=C(\nu_1)$ is a constant depending on $\nu_1$.
 \end{Lemma}
The following pointwise estimates will be used frequently in the next section.
 \begin{Lemma}\label{xuyao8DD8899}
 Let $u$ be a smooth function with sufficient decay at the spatial infinity. Then it holds that
\begin{align}\label{224iijuio}
\|u(t,\cdot)\|_{L^{\infty}(\mathbb{R})}&\leq C
E_1^{1/2}(u(t)),
\end{align}
\begin{align}\label{xuyao217}
\|\langle\xi\rangle^{1+\delta}u_{\xi}\|_{L_{x}^{{\infty}}(\mathbb{R})}+\|\langle \eta\rangle^{1+\delta}u_{\eta}\|_{L_{x}^{{\infty}}(\mathbb{R})}&\leq C
E_1^{1/2}(u(t))+C
E_2^{1/2}(u(t)),
\end{align}
\begin{align}\label{xuyao218}
&\|\langle\xi\rangle^{1+\delta}(u_{\xi\xi}+u_{\eta\xi})\|_{L_{x}^{{\infty}}(\mathbb{R})}
+\|\langle \eta\rangle^{1+\delta}(u_{\eta\eta}+u_{\xi\eta})\|_{L_{x}^{{\infty}}(\mathbb{R})}\nonumber\\
&\leq C
E^{1/2}(u(t))+C
\widetilde{E_3}^{1/2}(u(t)),
\end{align}
and
\begin{align}\label{xuyao219}
\|\langle\eta\rangle^{-\frac{1+\delta}{2}}\langle\xi\rangle^{1+\delta}u_{\xi}\|_{L^2_{s}L_{x}^{{\infty}}}+\|\langle\xi\rangle^{-\frac{1+\delta}{2}}\langle \eta\rangle^{1+\delta}u_{\eta}\|_{L^2_{s}L_{x}^{{\infty}}}&\leq C
\mathcal {E}_1^{1/2}(u(t))+C
\mathcal {E}_2^{1/2}(u(t)),
\end{align}
\begin{align}\label{xuyao220}
&\|\langle\eta\rangle^{-\frac{1+\delta}{2}}\langle\xi\rangle^{1+\delta}(u_{\xi\xi}+u_{\eta\xi})\|_{L^2_{s}L_{x}^{{\infty}}}
+\|\langle\xi\rangle^{-\frac{1+\delta}{2}}\langle \eta\rangle^{1+\delta}(u_{\eta\eta}+u_{\xi\eta})\|_{L^2_{s}L_{x}^{{\infty}}}\nonumber\\
&\leq C
\mathcal {E}^{1/2}(u(t))+C
\widetilde{\mathcal {E}_3}^{1/2}(u(t)).
\end{align}
 \end{Lemma}
 \begin{proof}
  It follows from the fundamental theorem of calculus and H\"{o}lder inequality that
\begin{align}
&\|u(t,\cdot)\|_{L^{\infty}(\mathbb{R})}\leq \|u_x(t,\cdot)\|_{L_{x}^{1}(\mathbb{R})}\leq \|u_{\xi}\|_{L_{x}^{1}(\mathbb{R})}+\|u_{\eta}\|_{L_{x}^{1}(\mathbb{R})}\nonumber\\
&\leq \|\langle \xi\rangle^{-1-\delta}\|_{L_{x}^{2}(\mathbb{R})}\|\langle \xi\rangle^{1+\delta}u_{\xi}\|_{L_{x}^{2}(\mathbb{R})}+\|\langle \eta\rangle^{-1-\delta}\|_{L_{x}^{2}(\mathbb{R})}\|\langle \eta\rangle^{1+\delta}u_{\eta}\|_{L_{x}^{2}(\mathbb{R})}\nonumber\\
&\leq C\big(\|\langle \xi\rangle^{1+\delta}u_{\xi}\|_{L_{x}^{2}(\mathbb{R})}+\|\langle \eta\rangle^{1+\delta}u_{\eta}\|_{L_{x}^{2}(\mathbb{R})}\big)\leq C
E_1^{1/2}(u(t)).
\end{align}
While \eqref{xuyao217}, \eqref{xuyao218}, \eqref{xuyao219} and \eqref{xuyao220} can be proved by Sobolev embedding $H^1(\mathbb{R})\hookrightarrow L^{\infty}(\mathbb{R})$ and the following fact
\begin{align}
|\partial_{x}\langle\xi\rangle^{1+\delta}|\leq C\langle\xi\rangle^{1+\delta},~ |\partial_{x}(\langle\eta\rangle^{-\frac{1+\delta}{2}}\langle\xi\rangle^{1+\delta})|\leq C \langle\eta\rangle^{-\frac{1+\delta}{2}}\langle\xi\rangle^{1+\delta},\\
|\partial_{x}\langle\eta\rangle^{1+\delta}|\leq C\langle\eta\rangle^{1+\delta},~ |\partial_{x}(\langle\xi\rangle^{-\frac{1+\delta}{2}}\langle\eta\rangle^{1+\delta})|\leq C \langle\xi\rangle^{-\frac{1+\delta}{2}}\langle\eta\rangle^{1+\delta}.
\end{align}
 \end{proof}
 The following lemma will paly a key role in the proof of global existence part of Theorem \ref{mainthm}.
 \begin{Lemma}\label{keyn9078p}
Let $u$ be a classical solution to the system \eqref{quasiwave} satisfying null conditions \eqref{order1}--\eqref{order4}. Assume that
\begin{align}\label{tyuu78888}
\varepsilon_1=\sup_{0\leq s\leq t}E^{1/2}(u(s))
\end{align}
is sufficiently small. Then we have
\begin{align}\label{xuyaoede333}
\sup_{0\leq s\leq t}\widetilde{E_3}(u(s))\leq C\sup_{0\leq s\leq t}E(u(s))
\end{align}
and
\begin{align}\label{xuyaojk9899}
\widetilde{\mathcal{E}_3}(u(t))\leq C\mathcal {E}(u(t)).
\end{align}
 \end{Lemma}
 \begin{proof}
From the system \eqref{quasiwave}, we have
 \begin{align}\label{xuyaojkk4555l}
 u_{\eta\eta\xi}&=A_1u_{\eta\eta\xi}
+A_2u_{\xi\eta\xi}+A_3u_{\eta\eta\eta}+\partial_{\eta}A_1u_{\eta\xi}
+\partial_{\eta}A_2u_{\xi\xi}+\partial_{\eta}A_3u_{\eta\eta}+\partial_{\eta}F
 \end{align}
 and
\begin{align}\label{dieryuuj}
 u_{\xi\xi\eta}&=A_1u_{\xi\xi\eta}
+A_2u_{\xi\xi\xi}+A_3u_{\xi\eta\eta}+\partial_{\xi}A_1u_{\xi\eta}
+\partial_{\xi}A_2u_{\xi\xi}+\partial_{\xi}A_3u_{\eta\eta}+\partial_{\xi}F.
 \end{align}
 \par
We first estimate ${\widetilde{E_3}}(u(t))$.
It follows from \eqref{xuyaojkk4555l}, \eqref{224iijuio}, Lemma \ref{xuyaouioo999} and Sobolev embedding $H^1(\mathbb{R})\hookrightarrow L^{\infty}(\mathbb{R})$ that
 \begin{align}\label{san12222}
&\|\langle \xi\rangle^{1+\delta}u_{\eta\eta\xi}\|_{L_{x}^2(\mathbb{R})}\nonumber\\
&\leq C\|A_1\|_{L_{x}^{\infty}(\mathbb{R})}\|\langle \xi\rangle^{1+\delta}u_{\eta\eta\xi}\|_{L_{x}^2(\mathbb{R})}+C\|A_2\|_{L_{x}^{\infty}(\mathbb{R})}\|\langle \xi\rangle^{1+\delta}u_{\xi\eta\xi}\|_{L_{x}^2(\mathbb{R})}\nonumber\\
&+C\|\langle \xi\rangle^{1+\delta}A_3\|_{L_{x}^{\infty}(\mathbb{R})}\|u_{\eta\eta\eta}\|_{L_{x}^2(\mathbb{R})}+C\|\partial_{\eta}A_1\|_{L_{x}^{\infty}(\mathbb{R})}\|\langle \xi\rangle^{1+\delta}u_{\eta\xi}\|_{L_{x}^2(\mathbb{R})}\nonumber\\
&+C\|\partial_{\eta}A_2\|_{L_{x}^{\infty}(\mathbb{R})}\|\langle \xi\rangle^{1+\delta}u_{\xi\xi}\|_{L_{x}^2(\mathbb{R})}
+C\|\langle \xi\rangle^{1+\delta}\partial_{\eta}A_3\|_{L_{x}^{2}(\mathbb{R})}\|u_{\eta\eta}\|_{L_{x}^{\infty}(\mathbb{R})}+\|\langle \xi\rangle^{1+\delta}\partial_{\eta}F\|_{L_{x}^2(\mathbb{R})}\nonumber\\
&\leq C\|u+u_{\xi}+u_{\eta}\|_{L_{x}^{\infty}(\mathbb{R})}\|\langle \xi\rangle^{1+\delta}u_{\eta\eta\xi}\|_{L_{x}^2(\mathbb{R})}+C\|u_{\eta}\|_{L_{x}^{\infty}(\mathbb{R})}\|\langle \xi\rangle^{1+\delta}u_{\xi\eta\xi}\|_{L_{x}^2(\mathbb{R})}\nonumber\\
&+C\|\langle \xi\rangle^{1+\delta}u_{\xi}\|_{L_{x}^{\infty}(\mathbb{R})}\|u_{\eta\eta\eta}\|_{L_{x}^2(\mathbb{R})}+C\|u_{\eta}+u_{\xi\eta}+u_{\eta\eta}\|_{L_{x}^{\infty}(\mathbb{R})}\|\langle \xi\rangle^{1+\delta}u_{\eta\xi}\|_{L_{x}^2(\mathbb{R})}\nonumber\\
&+C\|u_{\eta}+u_{\xi\eta}+u_{\eta\eta}\|_{L_{x}^{\infty}(\mathbb{R})}\|\langle \xi\rangle^{1+\delta}u_{\xi\xi}\|_{L_{x}^2(\mathbb{R})}\nonumber\\
&
+C\|\langle \xi\rangle^{1+\delta}|u_{\eta\xi}|+\langle \xi\rangle^{1+\delta}|u_{\xi}|(|u_{\eta}|+|u_{\eta\eta}|)\|_{L_{x}^2(\mathbb{R})}\|u_{\eta\eta}\|_{L_{x}^{\infty}(\mathbb{R})}
\nonumber\\
&+C\|u_{\eta}\|_{L_{x}^{\infty}(\mathbb{R})}\|\langle \xi\rangle^{1+\delta}u_{\eta\xi}\|_{L_{x}^2(\mathbb{R})}
+C\|\langle \xi\rangle^{1+\delta}u_{\xi}\|_{L_{x}^{\infty}(\mathbb{R})}\|u_{\eta\eta}\|_{L_{x}^2(\mathbb{R})}+C\|u_{\eta}\|_{L_{x}^{\infty}(\mathbb{R})}\|\langle \xi\rangle^{1+\delta}u_{\xi}\|_{L_{x}^2(\mathbb{R})}\nonumber\\
&\leq CE^{1/2}(u(t))\|\langle \xi\rangle^{1+\delta}u_{\eta\eta\xi}\|_{L_{x}^2(\mathbb{R})}+C\big(E(u(t))+E^{3/2}(u(t))\big).
 \end{align}
 Similarly, by \eqref{dieryuuj}, \eqref{224iijuio}, Lemma \ref{xuyaouioo999} and Sobolev embedding $H^1(\mathbb{R})\hookrightarrow L^{\infty}(\mathbb{R})$, we can also obtain
 \begin{align}\label{san1222ddd4442}
\|\langle \eta\rangle^{1+\delta}u_{\xi\xi\eta}\|_{L_{x}^2(\mathbb{R})}\leq CE^{1/2}(u(t))\|\langle \eta\rangle^{1+\delta}u_{\xi\xi\eta}\|_{L_{x}^2(\mathbb{R})}+C\big(E(u(t))+E^{3/2}(u(t))\big).
 \end{align}
 Thus, thanks to \eqref{san12222} and \eqref{san1222ddd4442}, we have
  \begin{align}\label{san1222ddd4448882}
{\widetilde{E_3}}^{1/2}(u(t))&\leq CE^{1/2}(u(t)){\widetilde{E_3}}^{1/2}(u(t))+C\big(E(u(t))+E^{3/2}(u(t))\big)\nonumber\\
&\leq C\varepsilon_1{\widetilde{E_3}}^{1/2}(u(t))+C\big(\varepsilon_1+\varepsilon_1^2)\big)E^{1/2}(u(t)).
 \end{align}
If $\varepsilon_1$ is sufficiently small, we can get
\eqref{xuyaoede333}.\par
Now we will estimate $\widetilde{\mathcal{E}_3}(u(t))$.
From \eqref{dieryuuj}, \eqref{224iijuio}, Lemma \ref{xuyaouioo999} and Sobolev embedding $H^1(\mathbb{R})\hookrightarrow L^{\infty}(\mathbb{R})$,  we can see
  \begin{align}\label{xyui9987}
&\|\langle\xi\rangle^{-\frac{1+\delta}{2}}\langle \eta\rangle^{1+\delta}u_{\xi\xi\eta}\|_{L_{s,x}^2}\nonumber\\
&\leq C\|A_1\|_{L_{s,x}^{\infty}}\|\langle\xi\rangle^{-\frac{1+\delta}{2}}\langle \eta\rangle^{1+\delta}u_{\xi\xi\eta}\|_{L_{s,x}^2}+C\|\langle\xi\rangle^{-\frac{1+\delta}{2}}\langle \eta\rangle^{1+\delta}A_2\|_{L_{s}^{2}L_{x}^{\infty}}\|u_{\xi\xi\xi}\|_{L_{s}^{\infty}L_{x}^{2}}\nonumber\\
&+C\|A_3\|_{L_{s,x}^{\infty}}\|\langle\xi\rangle^{-\frac{1+\delta}{2}}\langle \eta\rangle^{1+\delta}u_{\xi\eta\eta}\|_{L_{s,x}^2}+C\|\partial_{\xi}A_1\|_{L_{s,x}^{\infty}}\|\langle\xi\rangle^{-\frac{1+\delta}{2}}\langle \eta\rangle^{1+\delta}u_{\xi\eta}\|_{L_{s,x}^2}\nonumber\\
&+C\|\langle\xi\rangle^{-\frac{1+\delta}{2}}\langle \eta\rangle^{1+\delta}\partial_{\xi}A_2\|_{L_{s,x}^{2}}\|u_{\xi\xi}\|_{L_{s,x}^{\infty}}
+C\|\partial_{\xi}A_3\|_{L_{s,x}^{\infty}}\|\langle\xi\rangle^{-\frac{1+\delta}{2}}\langle \eta\rangle^{1+\delta}u_{\eta\eta}\|_{L_{s,x}^{2}}
\nonumber\\
&+\|\langle\xi\rangle^{-\frac{1+\delta}{2}}\langle \eta\rangle^{1+\delta}\partial_{\xi}F\|_{L_{s,x}^2}\nonumber\\
&\leq C\|u+u_{\xi}+u_{\eta}\|_{L_{s,x}^{\infty}}\|\langle\xi\rangle^{-\frac{1+\delta}{2}}\langle \eta\rangle^{1+\delta}u_{\xi\xi\eta}\|_{L_{s,x}^2}+C\|\langle\xi\rangle^{-\frac{1+\delta}{2}}\langle \eta\rangle^{1+\delta}u_{\eta}\|_{L_{s}^{2}L_{x}^{\infty}}\|u_{\xi\xi\xi}\|_{L_{s}^{\infty}L_{x}^{2}}\nonumber\\
&+C\|u_{\xi}\|_{L_{s,x}^{\infty}}\|\langle\xi\rangle^{-\frac{1+\delta}{2}}\langle \eta\rangle^{1+\delta}u_{\xi\eta\eta}\|_{L_{s,x}^2}+C\|u_{\xi}+u_{\xi\xi}+u_{\xi\eta}\|_{L_{s,x}^{\infty}}\|\langle\xi\rangle^{-\frac{1+\delta}{2}}\langle \eta\rangle^{1+\delta}u_{\xi\eta}\|_{L_{s,x}^2}\nonumber\\
&+C\|\langle\xi\rangle^{-\frac{1+\delta}{2}}\langle \eta\rangle^{1+\delta}(|u_{\xi\eta}|+|u_{\eta}|(|u_{\xi}|+|u_{\xi\xi}|))\|_{L_{s,x}^{2}}\|u_{\xi\xi}\|_{L_{s,x}^{\infty}}\nonumber\\
&
+C\|u_{\xi}+u_{\xi\xi}+u_{\xi\eta}\|_{L_{s,x}^{\infty}}\|\langle\xi\rangle^{-\frac{1+\delta}{2}}\langle \eta\rangle^{1+\delta}u_{\eta\eta}\|_{L_{s,x}^{2}}
\nonumber\\
&+C\|u_{\xi}\|_{L_{s,x}^{\infty}}\|\langle\xi\rangle^{-\frac{1+\delta}{2}}\langle \eta\rangle^{1+\delta}u_{\xi\eta}\|_{L_{s,x}^2}
+C\|\langle\xi\rangle^{-\frac{1+\delta}{2}}\langle \eta\rangle^{1+\delta}u_{\eta}\|_{L_{s}^{2}L_{x}^{\infty}}\|u_{\xi\xi}\|_{L_{s}^{\infty}L_{x}^{2}}\nonumber\\
&+C\|u_{\xi}\|_{L_{s,x}^{\infty}}\|\langle\xi\rangle^{-\frac{1+\delta}{2}}\langle \eta\rangle^{1+\delta}u_{\eta}\|_{L_{s,x}^2}\nonumber\\
&\leq C\sup_{0\leq s\leq t}E^{1/2}(u(s))\|\langle\xi\rangle^{-\frac{1+\delta}{2}}\langle \eta\rangle^{1+\delta}u_{\xi\xi\eta}\|_{L_{s,x}^2}\nonumber\\
&+C\big(\sup_{0\leq s\leq t}E^{1/2}(u(s))+\sup_{0\leq s\leq t}E(u(s))\big){{\mathcal{E}}}^{1/2}(u(t)).
 \end{align}
  Similarly, by \eqref{xuyaojkk4555l}, \eqref{224iijuio}, Lemma \ref{xuyaouioo999} and Sobolev embedding $H^1(\mathbb{R})\hookrightarrow L^{\infty}(\mathbb{R})$, we can also get
   \begin{align}\label{xyuiddd9987}
&\|\langle \eta\rangle^{-\frac{1+\delta}{2}}\langle \xi\rangle^{1+\delta}u_{\eta\eta\xi}\|_{L_{s,x}^2}\nonumber\\
&\leq C\sup_{0\leq s\leq t}E^{1/2}(u(s))\|\langle \eta\rangle^{-\frac{1+\delta}{2}}\langle \xi\rangle^{1+\delta}u_{\eta\eta\xi}\|_{L_{s,x}^2}\nonumber\\
&+C\big(\sup_{0\leq s\leq t}E^{1/2}(u(s))+\sup_{0\leq s\leq t}E(u(s))\big){{\mathcal{E}}}^{1/2}(u(t)).
 \end{align}
 By \eqref{xyui9987} and \eqref{xyuiddd9987} we can obtain
  \begin{align}
&{\widetilde{\mathcal{E}_3}}^{1/2}(u(t))\nonumber\\
&\leq C\sup_{0\leq s\leq t}E^{1/2}(u(s)){\widetilde{\mathcal{E}_3}}^{1/2}(u(t))+C\big(\sup_{0\leq s\leq t}E^{1/2}(u(s))+\sup_{0\leq s\leq t}E(u(s))\big){{\mathcal{E}}}^{1/2}(u(t))\nonumber\\
&\leq C\varepsilon_1{\widetilde{\mathcal{E}_3}}^{1/2}(u(t))+C\big(\varepsilon_1+\varepsilon_1^2\big){{\mathcal{E}}}^{1/2}(u(t)).
 \end{align}
 If $\varepsilon_1$ is sufficiently small, we can get
\eqref{xuyaojk9899}.
 \end{proof}
 The following lemma will be used in the proof of asymptotic behavior part of Theorem \ref{mainthm}. For the proof of this lemma, we refer the reader to Lemma {\rm{6.12}} in Katayama {\rm {\cite{MR3729247}}}, where high space dimension case is also considered.
 \begin{Lemma}\label{xuyaoodddd897}
 If $G\in L^1(\mathbb{R}^{+};L^2(\mathbb{R}))$, i.e.,
 \begin{align}\label{xuyo8fggg970}
 \int_0^{+\infty}\|G(t,\cdot)\|_{L^2(\mathbb{R})}dt<+\infty,
 \end{align}
Then the global solution to
 \begin{align}
 u_{\xi\eta}=G
 \end{align}
is asymptotically free in the energy sense.
 \end{Lemma}
 Finally, for the convenience, we will introduce some notations concerning the weight functions used in the weighted energy estimates in the next section.
Fix $0<\delta<1$. Set
\begin{align}\label{xuyao00099}
\phi(x)=\langle x\rangle^{2+2\delta}.
\end{align}
It is easy to verify that
\begin{align}\label{20p06}
|\phi'(x)|\leq 4\langle x\rangle^{1+2\delta}.
\end{align}
Let
\begin{align}
q(x)=\int_{-\infty}^{x}{\langle \rho\rangle^{-(1+\delta)}}d\rho
\end{align}
and
\begin{align}
\psi(x)=e^{-q(x)}.
\end{align}
We can verify that
\begin{align}
\psi'(x)=-\psi(x)\langle x\rangle^{-(1+\delta)}.
\end{align}
We note that there exists a positive constant $c$ such that
\begin{align}\label{rg56}
c^{-1}\leq \psi(x)\leq c,
\end{align}
thus it holds that
\begin{align}\label{here456}
c^{-1}\langle x\rangle^{-(1+\delta)}\leq -\psi'(x)\leq c\langle x\rangle^{-(1+\delta)}.
\end{align}
\section{Proof of Theorem \ref{mainthm}}\label{xhzuyo898}
Now we will prove Theorem \ref{mainthm}. We first prove the global existence part of Theorem \ref{mainthm} by some bootstrap argument. Assume that $u$ is a classical solution to the Cauchy problem \eqref{quasiwave}--\eqref{hj89i}. We will show that there exist positive constants $\varepsilon_0$ and $A$ such that
\begin{align}\label{rgty6788888}
\sup_{0\leq s\leq t}E(u(s))+\mathcal {E}(u(t))\leq A^2\varepsilon^2
\end{align}
 under
the assumption
\begin{align}\label{dddddd}
\sup_{0\leq s\leq t}E(u(s))+\mathcal {E}(u(t))\leq 4A^2\varepsilon^2,
\end{align}
where $0<\varepsilon\leq \varepsilon_0$.
Then based on  estimate \eqref{rgty6788888} and Lemma \ref{xuyaoodddd897},
the proof of asymptotic behavior part of Theorem \ref{mainthm} can also be given.
\subsection{Low order energy estimates}
We first estimate $\sup_{0\leq s\leq t}E_2(u(s))$ and $\mathcal {E}_2(u(t))$. $\sup_{0\leq s\leq t}E_1(u(s))$ and $\mathcal {E}_1(u(t))$ can be estimated similarly. \par
For any multi-index $a=(a_1,a_2), |a|=1$, note that $Z^{a}u$ satisfies
\begin{align}\label{quasiwave8989}
Z^{a}u_{\xi\eta}=A_1Z^{a}u_{\xi\eta}+A_2Z^{a}u_{\xi\xi}+A_3Z^{a}u_{\eta\eta}+Z^{a}A_1u_{\xi\eta}+Z^{a}A_2u_{\xi\xi}+Z^{a}A_3u_{\eta\eta}+Z^{a}F.
\end{align}
Multiply $2\psi(\eta)\phi(\xi)Z^au^{T}_{\xi}$ on both sides of \eqref{quasiwave8989}. Noting the symmetry of $A_1$, by Leibniz's rule we have
\begin{align}\label{xhjyio999}
&\big(\psi(\eta)\phi(\xi)|Z^{a}u_{\xi}|^2\big)_{\eta}-\psi'(\eta)\phi(\xi)|Z^{a}u_{\xi}|^2\nonumber\\
&=\big(\psi(\eta)\phi(\xi)Z^{a}u_{\xi}^{{T}}A_1Z^{a}u_{\xi}\big)_{\eta}-\psi'(\eta)\phi(\xi)Z^{a}u_{\xi}^{{T}}A_1Z^{a}u_{\xi}\nonumber\\
&-\psi(\eta)\phi(\xi)Z^{a}u_{\xi}^{{T}}\partial_{\eta}A_1Z^{a}u_{\xi}
+2\psi(\eta)\phi(\xi)Z^{a}u^{T}_{\xi}G_a,
\end{align}
where
\begin{align}\label{xuyaodeii99}
G_a&=A_2Z^{a}u_{\xi\xi}+A_3Z^{a}u_{\eta\eta}+Z^{a}A_1u_{\xi\eta}+Z^{a}A_2u_{\xi\xi}+Z^{a}A_3u_{\eta\eta}+Z^{a}F.
\end{align}
Similarly, multiply $2\psi(\xi)\phi(\eta)Z^{a}u^{{T}}_{\eta}$ on both sides of \eqref{quasiwave}. The symmetry of $A_1$ and Leibniz's rule also imply
\begin{align}\label{xhjyi7666o}
&\big(\psi(\xi)\phi(\eta)|Z^{a}u_{\eta}|^2\big)_{\xi}-\psi'(\xi)\phi(\eta)|Z^{a}u_{\eta}|^2\nonumber\\
&=\big(\psi(\xi)\phi(\eta)Z^{a}u_{\eta}^{{T}}A_1Z^{a}u_{\eta}\big)_{\xi}-\psi'(\xi)\phi(\eta)Z^{a}u_{\eta}^{{T}}A_1Z^{a}u_{\eta}\nonumber\\
&-\psi(\xi)\phi(\eta)Z^{a}u_{\eta}^{{T}}\partial_{\xi}A_1Z^{a}u_{\eta}+2\psi(\xi)\phi(\eta)Z^{a}u^{{T}}_{\eta}{G_a}.
\end{align}
\par
Integrating on $[0,t]\times\mathbb{R}$ on both sides of \eqref{xhjyio999} and \eqref{xhjyi7666o}, we conclude from the fundamental theorem of calculus that
\begin{align}\label{hj78999888}
&\int_{\mathbb{R}}\big(e_2(t,x)+\widetilde{e}_2(t,x)\big)dx+\int_0^{t}\!\!\int_{\mathbb{R}}p_2(s,x)dxds\nonumber\\
&=\int_{\mathbb{R}}\big(e_2(0,x)+\widetilde{e}_2(0,x)\big)dx+\int_0^{t}\!\!\int_{\mathbb{R}}q_2(s,x)dxds\nonumber\\
&+2\int_0^{t}\!\!\int_{\mathbb{R}}\psi(\eta)\phi(\xi)Z^{a}u^{T}_{\xi}G_adxds+2\int_0^{t}\!\!\int_{\mathbb{R}}\psi(\xi)\phi(\eta)Z^{a}u^{{T}}_{\eta}{G_a}dxds,
\end{align}
where
\begin{align}
e_2=\psi(\eta)\phi(\xi)|Z^{a}u_{\xi}|^2+\psi(\xi)\phi(\eta)|Z^{a}u_{\eta}|^2,
\end{align}
\begin{align}
\widetilde{e}_2=-\psi(\eta)\phi(\xi)Z^{a}u_{\xi}^{{T}}A_1Z^{a}u_{\xi}
-\psi(\xi)\phi(\eta)Z^{a}u_{\eta}^{{T}}A_1Z^{a}u_{\eta},
\end{align}
\begin{align}
p_2=-\psi'(\eta)\phi(\xi)|Z^{a}u_{\xi}|^2-\psi'(\xi)\phi(\eta)|Z^{a}u_{\eta}|^2
\end{align}
and
\begin{align}
q_2&=-\psi'(\eta)\phi(\xi)Z^{a}u_{\xi}^{{T}}A_1Z^{a}u_{\xi}
-\psi(\eta)\phi(\xi)Z^{a}u_{\xi}^{{T}}\partial_{\eta}A_1Z^{a}u_{\xi}\nonumber\\
&-\psi'(\xi)\phi(\eta)Z^{a}u_{\eta}^{{T}}A_1Z^{a}u_{\eta}
-\psi(\xi)\phi(\eta)Z^{a}u_{\eta}^{{T}}\partial_{\xi}A_1Z^{a}u_{\eta}.
\end{align}
\par
In view of \eqref{xuyao00099} and \eqref{rg56}, we can see
\begin{align}
c^{-1}e_2(t,x)\leq |\langle \xi\rangle^{1+\delta}Z^{a}u_{\xi}|^2+|\langle \eta\rangle^{1+\delta}Z^{a}u_{\eta}|^2\leq ce_2(t,x).
\end{align}
It follows from \eqref{xuyao00099}, \eqref{rg56}, \eqref{order1}, \eqref{224iijuio} and Sobolev embedding $H^1(\mathbb{R})\hookrightarrow L^{\infty}(\mathbb{R})$ that
\begin{align}
|\widetilde{e}_2(t,x)|&\leq C\langle \xi\rangle^{2+2\delta}|Z^{a}u_{\xi}^{{T}}A_1Z^{a}u_{\xi}|
+C\langle \eta\rangle^{2+2\delta}|Z^{a}u_{\eta}^{{T}}A_1Z^{a}u_{\eta}|\nonumber\\
&\leq C|\langle \xi\rangle^{1+\delta}Z^{a}u_{\xi}|^2(|u|+|u_{\xi}|+|u_{\eta}|)
+C|\langle \eta\rangle^{1+\delta}Z^{a}u_{\eta}|^2(|u|+|u_{\xi}|+|u_{\eta}|)\nonumber\\
&\leq C\big(E^{1/2}_{1}(u(t))+E^{1/2}_{2}(u(t))\big) \big(|\langle \xi\rangle^{1+\delta}Z^{a}u_{\xi}|^2+ |\langle \eta\rangle^{1+\delta}Z^{a}u_{\eta}|^2\big)\nonumber\\
&\leq CE^{1/2}(u(t)) \big(|\langle \xi\rangle^{1+\delta}Z^{a}u_{\xi}|^2+ |\langle \eta\rangle^{1+\delta}Z^{a}u_{\eta}|^2\big).
\end{align}
Thus, for small solutions, we can get
\begin{align}\label{siof00766hm}
(2c)^{-1}\big(e_2(t,x)+\widetilde{e}_2(t,x)\big)\leq |\langle \xi\rangle^{1+\delta}Z^{a}u_{\xi}|^2+|\langle \eta\rangle^{1+\delta}Z^{a}u_{\eta}|^2\leq 2c\big(e_2(t,x)+\widetilde{e}_2(t,x)\big).
\end{align}
From \eqref{xuyao00099} and \eqref{here456}, it follows that
\begin{align}\label{opp9pppp888}
c^{-1}p_2(t,x)\leq |\langle\eta\rangle^{-\frac{1+\delta}{2}}\langle \xi\rangle^{1+\delta}Z^{a}u_{\xi}|^2+|\langle\xi\rangle^{-\frac{1+\delta}{2}}\langle \eta\rangle^{1+\delta}Z^{a}u_{\eta}|^2\leq cp_2(t,x).
\end{align}
According to \eqref{xuyao00099}, \eqref{rg56}, \eqref{here456}, \eqref{order1} and Lemma \ref{xuyaouioo999}, we have
\begin{align}\label{qguji}
|q_2(t,x)|&\leq C\langle \eta\rangle^{-(1+\delta)}\langle \xi\rangle^{2+2\delta}|Z^{a}u_{\xi}|^2(|u|+|u_{\xi}|+|u_{\eta}|)+C\langle \xi\rangle^{2+2\delta}|Z^{a}u_{\xi}|^2(|u_{\eta}|+|Zu_{\eta}|)\nonumber\\
&+ C\langle \xi\rangle^{-(1+\delta)}\langle \eta\rangle^{2+2\delta}|Z^{a}u_{\eta}|^2(|u|+|u_{\xi}|+|u_{\eta}|)+C\langle \eta\rangle^{2+2\delta}|Z^{a}u_{\eta}|^2(|u_{\xi}|+|Zu_{\xi}|)\nonumber\\
&\leq C|\langle \eta\rangle^{-\frac{1+\delta}{2}}\langle \xi\rangle^{1+\delta}Z^{a}u_{\xi}|^2(|u|+|u_{\xi}|+|u_{\eta}|)\nonumber\\
&+C|\langle \eta\rangle^{-\frac{1+\delta}{2}}\langle \xi\rangle^{1+\delta}Z^{a}u_{\xi}|^2(|\langle \eta\rangle^{1+\delta}u_{\eta}|+|\langle \eta\rangle^{1+\delta} Zu_{\eta}|)\nonumber\\
&+ C|\langle \xi\rangle^{-\frac{1+\delta}{2}}\langle \eta\rangle^{1+\delta}Z^{a}u_{\eta}|^2(|u|+|u_{\xi}|+|u_{\eta}|)\nonumber\\
&+C|\langle \xi\rangle^{-\frac{1+\delta}{2}}\langle \eta\rangle^{1+\delta}Z^{a}u_{\eta}|^2(|\langle \xi\rangle^{1+\delta}u_{\xi}|+|\langle \xi\rangle^{1+\delta} Zu_{\xi}|).
\end{align}
By \eqref{xuyaodeii99}, \eqref{order22}, \eqref{order33} and Lemma \ref{xuyaouioo99999999}, we can obtain
\begin{align}\label{pointwiseG}
|G_{a}|&\leq C|Z^{a}u_{\xi\xi}||u_{\eta}|+C|Z^{a}u_{\eta\eta}||u_{\xi}|+C(|u_{\xi}|+|u_{\eta}|+|Z^au_{\xi}|+|Z^au_{\eta}|)|u_{\xi\eta}|\nonumber\\
&+C(|u_{\eta}|+|Z^au_{\eta}|)|u_{\xi\xi}|+C(|u_{\xi}|+|Z^au_{\xi}|)|u_{\eta\eta}|\nonumber\\
&+C(|u_{\xi}||u_{\eta}|+|Z^{a}u_{\xi}||u_{\eta}|+|Z^{a}u_{\eta}||u_{\xi}|).
\end{align}
\par
Thanks to \eqref{hj78999888}, \eqref{siof00766hm} and \eqref{opp9pppp888}, we see
\begin{align}\label{xddd89uii888}
&\sup_{0\leq s\leq t}E_2(u(s))+ \mathcal{E}_2(u(t))\nonumber\\
&\leq C E_2(u(0))+C\sum_{|a|=1}\int_0^{t}\|q_2(s,\cdot)\|_{L^1(\mathbb{R})}ds\nonumber\\
&+C\sum_{|a|=1}\int_0^{t}\|\langle \xi\rangle^{2+2\delta}Z^{a}u_{\xi}G_a\|_{L_{x}^1(\mathbb{R})}ds+C\sum_{|a|=1}\int_0^{t}\|\langle \eta\rangle^{2+2\delta}Z^{a}u_{\eta}{G_a}\|_{L_{x}^1(\mathbb{R})}ds.
\end{align}
It follows from \eqref{qguji}, H\"{o}lder inequality, Lemma \ref{xuyao8DD8899}, Sobolev embedding $H^1(\mathbb{R})\hookrightarrow L^{\infty}(\mathbb{R})$ and Lemma \ref{keyn9078p}
 that
\begin{align}\label{qguji8900}
&\int_0^{t}\|q_2(s,\cdot)\|_{L^1(\mathbb{R})}ds\nonumber\\
&\leq C\|\langle \eta\rangle^{-\frac{1+\delta}{2}}\langle \xi\rangle^{1+\delta}Z^{a}u_{\xi}\|_{L^2_{s,x}}^2(\|u\|_{L^{\infty}_{s,x}}+\|u_{\xi}\|_{L^{\infty}_{s,x}}+\|u_{\eta}\|_{L^{\infty}_{s,x}})\nonumber\\
&+C\|\langle \eta\rangle^{-\frac{1+\delta}{2}}\langle \xi\rangle^{1+\delta}Z^{a}u_{\xi}\|_{L^2_{s,x}}^2(\|\langle \eta\rangle^{1+\delta}u_{\eta}\|_{L^{\infty}_{s,x}}+\|\langle \eta\rangle^{1+\delta} Zu_{\eta}\|_{L^{\infty}_{s,x}})\nonumber\\
&+ C\|\langle \xi\rangle^{-\frac{1+\delta}{2}}\langle \eta\rangle^{1+\delta}Z^{a}u_{\eta}\|_{L^2_{s,x}}^2(\|u\|_{L^{\infty}_{s,x}}+\|u_{\xi}\|_{L^{\infty}_{s,x}}+\|u_{\eta}\|_{L^{\infty}_{s,x}})\nonumber\\
&+C\|\langle \xi\rangle^{-\frac{1+\delta}{2}}\langle \eta\rangle^{1+\delta}Z^{a}u_{\eta}\|_{L^{2}_{s,x}}^2(\|\langle \xi\rangle^{1+\delta}u_{\xi}\|_{L^{\infty}_{s,x}}+\|\langle \xi\rangle^{1+\delta} Zu_{\xi}\|_{L^{\infty}_{s,x}})\nonumber\\
&\leq C\big(\sup_{0\leq s\leq t}{E^{1/2}}(u(s))+\sup_{0\leq s\leq t}{\widetilde{E_3}}^{1/2}(u(s))\big)\mathcal {E}_2(u(t))\nonumber\\
&\leq C\sup_{0\leq s\leq t}{E^{1/2}}(u(s))\mathcal {E}_2(u(t)).
\end{align}
By H\"{o}lder inequality, it is easy to see that
\begin{align}\label{xuyaoo8967990}
&\int_0^{t}\|\langle \xi\rangle^{2+2\delta}Z^{a}u_{\xi}G_a\|_{L_{x}^1(\mathbb{R})}ds\nonumber\\
&\leq C\|\langle \eta\rangle^{-\frac{1+\delta}{2}}\langle \xi\rangle^{1+\delta}
Z^{a}u_{\xi}\|_{L^2_{s,x}}\|\langle \eta\rangle^{\frac{1+\delta}{2}}\langle \xi\rangle^{1+\delta}
G_a\|_{L^2_{s,x}}\nonumber\\
&\leq \mathcal {E}_2^{1/2}(u(t))\|\langle \eta\rangle^{\frac{1+\delta}{2}}\langle \xi\rangle^{1+\delta}
G_a\|_{L^2_{s,x}}.
\end{align}
By \eqref{pointwiseG}, Lemma \ref{xuyao8DD8899} and Lemma \ref{keyn9078p}, we can get
%
\begin{align}\label{dddf678899}
&\|\langle \eta\rangle^{\frac{1+\delta}{2}}\langle \xi\rangle^{1+\delta}
G_a\|_{L^2_{s,x}}\nonumber\\
&\leq C\|\langle \eta\rangle^{-\frac{1+\delta}{2}}\langle \xi\rangle^{1+\delta}Z^{a}u_{\xi\xi}\|_{L^2_{s,x}}\|\langle \eta\rangle^{1+\delta}u_{\eta}\|_{L^{\infty}_{s,x}}\nonumber\\
&+C\|\langle \eta\rangle^{1+\delta}Z^{a}u_{\eta\eta}\|_{L^{\infty}_{s}L^{2}_{x}}\|\langle \eta\rangle^{-\frac{1+\delta}{2}}\langle \xi\rangle^{1+\delta}u_{\xi}\|_{L^{2}_{s}L^{\infty}_{x}}\nonumber\\
&+C\|\langle \eta\rangle^{-\frac{1+\delta}{2}}\langle \xi\rangle^{1+\delta}(|u_{\xi}|+|Z^au_{\xi}|)\|_{L^2_{s,x}}\|\langle \eta\rangle^{1+\delta}u_{\xi\eta}\|_{L^{\infty}_{s,x}}\nonumber\\
&+C\||\langle \eta\rangle^{1+\delta}(|u_{\eta}|+|Z^au_{\eta}|)|\|_{L^{\infty}_{s,x}}\|\langle \eta\rangle^{-\frac{1+\delta}{2}}\langle \xi\rangle^{1+\delta}u_{\eta\xi}\|_{L^2_{s,x}}\nonumber\\
&+C\|\langle \eta\rangle^{1+\delta}(|u_{\eta}|+|Z^au_{\eta}|)\|_{L^{\infty}_{s,x}} \|\langle \eta\rangle^{-\frac{1+\delta}{2}}\langle \xi\rangle^{1+\delta}
u_{\xi\xi}\|_{L^2_{s,x}}\nonumber\\
&+C\|\langle \eta\rangle^{-\frac{1+\delta}{2}}\langle \xi\rangle^{1+\delta}
(|u_{\xi}|+|Z^au_{\xi}|)\|_{L^2_{s,x}}\|\langle \eta\rangle^{1+\delta}u_{\eta\eta}\|_{L^{\infty}_{s,x}}\nonumber\\
&+C\|\langle \eta\rangle^{1+\delta}(|u_{\eta}|+|Z^au_{\eta}|)\|_{L^2_{s,x}} \|\langle \eta\rangle^{-\frac{1+\delta}{2}}\langle \xi\rangle^{1+\delta}
u_{\xi}\|_{L^{\infty}_{s,x}}\nonumber\\
&+C\|\langle \eta\rangle^{-\frac{1+\delta}{2}}\langle \xi\rangle^{1+\delta}
(|u_{\xi}|+|Z^au_{\xi}|)\|_{L^2_{s,x}}\|\langle \eta\rangle^{1+\delta}u_{\eta}\|_{L^{\infty}_{s,x}}\nonumber\\
&\leq C\big(\sup_{0\leq s\leq t}{E^{1/2}}(u(s))+\sup_{0\leq s\leq t}{\widetilde{E_3}}^{1/2}(u(s))\big)\mathcal {E}^{1/2}_3(u(t))\nonumber\\
&\leq C\sup_{0\leq s\leq t}{E^{1/2}}(u(s))\mathcal {E}^{1/2}_3(u(t)).
\end{align}
It follows from \eqref{xuyaoo8967990} and \eqref{dddf678899} that
\begin{align}\label{xuyaoo896799988880}
\int_0^{t}\|\langle \xi\rangle^{2+2\delta}Z^{a}u_{\xi}G_a\|_{L_{x}^1(\mathbb{R})}ds
&\leq  C\sup_{0\leq s\leq t}{E^{1/2}}(u(s))\mathcal {E}_3(u(t)).
\end{align}
Similarly, we can also get
\begin{align}\label{xuyaoo876yuu96799988880}
\int_0^{t}\|\langle \eta\rangle^{2+2\delta}Z^{a}u_{\eta}{G_a}\|_{L_{x}^1(\mathbb{R})}ds
&\leq  C\sup_{0\leq s\leq t}{E^{1/2}}(u(s))\mathcal {E}_3(u(t)).
\end{align}
Consequently, the combination of \eqref{xddd89uii888}, \eqref{qguji8900}, \eqref{xuyaoo896799988880} and \eqref{xuyaoo876yuu96799988880} implies
\begin{align}\label{dijie3455888}
\sup_{0\leq s\leq t}E_2(u(s))+ \mathcal{E}_2(u(t))
\leq C E_2(u(0))+C\sup_{0\leq s\leq t}{E^{1/2}}(u(s))\mathcal {E}(u(t)).
\end{align}
\par
By a similar (but simpler) argument, we can also show
\begin{align}\label{dijie3455}
\sup_{0\leq s\leq t}E_1(u(s))+ \mathcal{E}_1(u(t))
\leq C E_1(u(0))+C\sup_{0\leq s\leq t}{E^{1/2}}(u(s))\mathcal {E}(u(t)).
\end{align}
\par
\subsection{High order energy estimates}
Now we will estimate  $\sup_{0\leq s\leq t}E_3(u(s))$ and $\mathcal {E}_3(u(t))$.\par
 For any multi-index $a=(a_1,a_2)$, $|a|=2, a_1\neq 0$, Leibniz's rule gives
\begin{align}\label{quasiwavegao7888}
Z^{a}u_{\xi\eta}&=A_1Z^{a}u_{\xi\eta}+
+A_2Z^{a}u_{\xi\xi}+A_3Z^{a}u_{\eta\eta}+H_a,
\end{align}
where
\begin{align}\label{ghjuiioo99}
H_a&=\sum_{\substack{c+d=a\\ c\neq 0}}\lambda_{cd}Z^{c}A_1Z^{d}u_{\xi\eta}+\sum_{\substack{c+d=a\\ c\neq 0}}\lambda_{cd}Z^{c}A_2Z^{d}u_{\xi\xi}+\sum_{\substack{c+d=a\\ c\neq 0}}\lambda_{cd}Z^{c}A_3Z^{d}u_{\eta\eta}+Z^aF,
\end{align}
$\lambda_{cd}$ are some constants.
Multiply $2\psi(\eta)\phi(\xi)Z^{a}u^{T}_{\xi}$ on both sides of \eqref{quasiwavegao7888}. Noting the symmetry of $A_1,A_2$ and $A_3$, by Leibniz's rule we can get
\begin{align}\label{lemkey17888}
&\big(\psi(\eta)\phi(\xi)|Z^{a}u_{\xi}|^2\big)_{\eta}-\psi'(\eta)\phi(\xi)|Z^{a}u_{\xi}|^2\nonumber\\
&=\big(\psi(\eta)\phi(\xi)Z^{a}u_{\xi}^{{T}}A_1Z^{a}u_{\xi}\big)_{\eta}-\psi'(\eta)\phi(\xi)Z^{a}u_{\xi}^{{T}}A_1Z^{a}u_{\xi}
-\psi(\eta)\phi(\xi)Z^{a}u_{\xi}^{{T}}\partial_{\eta}A_1Z^{a}u_{\xi}\nonumber\\
&+\big(\psi(\eta)\phi(\xi)Z^{a}u_{\xi}^{{T}}A_2Z^{a}u_{\xi}\big)_{\xi}-\psi(\eta)\phi'(\xi)Z^{a}u_{\xi}^{{T}}A_2Z^{a}u_{\xi}
-\psi(\eta)\phi(\xi)Z^{a}u_{\xi}^{{T}}\partial_{\xi}A_2Z^{a}u_{\xi}\nonumber\\
&+\big(2\psi(\eta)\phi(\xi)Z^{a}u_{\xi}^{{T}}A_3Z^{a}u_{\eta}\big)_{\eta}-2\psi'(\eta)\phi(\xi)Z^{a}u_{\xi}^{{T}}A_3Z^{a}u_{\eta}
-2\psi(\eta)\phi(\xi)Z^{a}u_{\xi}^{{T}}\partial_{\eta}A_3Z^{a}u_{\eta}\nonumber\\
&-\big(\psi(\eta)\phi(\xi)Z^{a}u_{\eta}^{{T}}A_3Z^{a}u_{\eta}\big)_{\xi}+\psi(\eta)\phi'(\xi)Z^{a}u_{\eta}^{{T}}A_3Z^{a}u_{\eta}
+\psi(\eta)\phi(\xi)Z^{a}u_{\eta}^{{T}}\partial_{\xi}A_3Z^{a}u_{\eta}\nonumber\\
&+2\psi(\eta)\phi(\xi)Z^{a}u_{\xi}^{{T}}H_a.
\end{align}
\par
Integrating on $[0,t]\times\mathbb{R}$ on both sides of \eqref{lemkey17888}, by the fundamental theorem of calculus, we have
\begin{align}\label{hj78999}
&\int_{\mathbb{R}}e_3(t,x)dx+\int_0^{t}\!\!\int_{\mathbb{R}}p_3(s,x)dxds\nonumber\\
&=\int_{\mathbb{R}}\big(e_3(0,x)+\widetilde{e}_3(0,x)\big)dx-\int_{\mathbb{R}}\widetilde{e}_3(t,x)dx\int_0^{t}\!\!\int_{\mathbb{R}}q_3(s,x)dxds\nonumber\\
&+2\int_0^{t}\!\!\int_{\mathbb{R}}\psi(\eta)\phi(\xi)Z^{a}u^{{T}}_{\xi}H_adxds,
\end{align}
where
\begin{align}
e_3=\psi(\eta)\phi(\xi)|Z^{a}u_{\xi}|^2,~~p_3=-\psi'(\eta)\phi(\xi)|Z^{a}u_{\xi}|^2,
\end{align}
\begin{align}
\widetilde{e}_3&=-\psi(\eta)\phi(\xi)Z^{a}u_{\xi}^{{T}}A_1Z^{a}u_{\xi}-\psi(\eta)\phi(\xi)Z^{a}u_{\xi}^{{T}}A_2Z^{a}u_{\xi}\nonumber\\
&
-2\psi(\eta)\phi(\xi)Z^{a}u_{\xi}^{{T}}A_3Z^{a}u_{\eta}+\psi(\eta)\phi(\xi)Z^{a}u_{\eta}^{{T}}A_3Z^{a}u_{\eta}
\end{align}
and
\begin{align}
q_3&=-\psi'(\eta)\phi(\xi)Z^{a}u_{\xi}^{{T}}A_1Z^{a}u_{\xi}
-\psi(\eta)\phi(\xi)Z^{a}u_{\xi}^{{T}}\partial_{\eta}A_1Z^{a}u_{\xi}\nonumber\\
&
-\psi(\eta)\phi'(\xi)Z^{a}u_{\xi}^{{T}}A_2Z^{a}u_{\xi}
-\psi(\eta)\phi(\xi)Z^{a}u_{\xi}^{{T}}\partial_{\xi}A_2Z^{a}u_{\xi}\nonumber\\
&-2\psi'(\eta)\phi(\xi)Z^{a}u_{\xi}^{{T}}A_3Z^{a}u_{\eta}
-2\psi(\eta)\phi(\xi)Z^{a}u_{\xi}^{{T}}\partial_{\eta}A_3Z^{a}u_{\eta}\nonumber\\
&
+\psi(\eta)\phi'(\xi)Z^{a}u_{\eta}^{{T}}A_3Z^{a}u_{\eta}
+\psi(\eta)\phi(\xi)Z^{a}u_{\eta}^{{T}}\partial_{\xi}A_3Z^{a}u_{\eta}.
\end{align}
\par
In view of \eqref{xuyao00099} and \eqref{rg56}, we have
\begin{align}\label{333333}
c^{-1}e_3(t,x)\leq |\langle \xi\rangle^{1+\delta}Z^{a}u_{\xi}|^2\leq ce_3(t,x).
\end{align}
By \eqref{xuyao00099} and \eqref{here456}, we can obtain
\begin{align}\label{opp9pppp}
c^{-1}p_3(t,x)\leq |\langle\eta\rangle^{-\frac{1+\delta}{2}}\langle \xi\rangle^{1+\delta}Z^{a}u_{\xi}|^2\leq cp_3(t,x).
\end{align}
It follows from \eqref{xuyao00099}, \eqref{rg56}, \eqref{order1}, \eqref{order22} and \eqref{order33} that
\begin{align}\label{By3555}
|\widetilde{e}_3(t,x)|&\leq C\langle \xi\rangle^{2+2\delta}|Z^{a}u_{\xi}^{{T}}A_1Z^{a}u_{\xi}|+C\langle \xi\rangle^{2+2\delta}|Z^{a}u_{\xi}^{{T}}A_2Z^{a}u_{\xi}|
\nonumber\\
&+C\langle \xi\rangle^{2+2\delta}|Z^{a}u_{\xi}^{{T}}A_3Z^{a}u_{\eta}|+C\langle \xi\rangle^{2+2\delta}|Z^{a}u_{\eta}^{{T}}A_3Z^{a}u_{\eta}|\nonumber\\
&\leq C|\langle \xi\rangle^{1+\delta}Z^{a}u_{\xi}|^2(|u|+|u_{\xi}|+|u_{\eta}|)\nonumber\\
&+C|\langle \xi\rangle^{1+\delta}Z^{a}u_{\xi}||\langle \xi\rangle^{1+\delta}u_{\xi}||Z^{a}u_{\eta}|+{\textcolor{blue}{C|\langle \xi\rangle^{1+\delta}Z^{a}u_{\eta}|}}|\langle \xi\rangle^{1+\delta}u_{\xi}||Z^{a}u_{\eta}|\nonumber\\
&\leq C|\langle \xi\rangle^{1+\delta}Z^{a}u_{\xi}|^2(|u|+|u_{\xi}|+|u_{\eta}|)+C\sum_{|b|=2}|\langle \xi\rangle^{1+\delta}Z^{b}u_{\xi}||\langle \xi\rangle^{1+\delta}u_{\xi}||Z^{a}u_{\eta}|.
 \end{align}
Here we also use the following simple but important fact: for the multi-index $a=(a_1,a_2),|a|=2,a_1\neq 0$, it holds that
\begin{align}\label{fact}
|\langle \xi\rangle^{1+\delta}Z^{a}u_{\eta}|=|\langle \xi\rangle^{1+\delta}\partial_{\xi}^{a_1}\partial_{\eta}^{a_2}u_{\eta}|\leq \sum_{|b|=2}|\langle \xi\rangle^{1+\delta}Z^{b}u_{\xi}|.
\end{align}
According to \eqref{order1}, \eqref{order22}, \eqref{order33}, \eqref{xuyao00099}, \eqref{20p06}, \eqref{rg56}, \eqref{here456}, Lemma \ref{xuyaouioo99999999} and \eqref{fact}, we can get
\begin{align}\label{qguji999980877778}
|q_3(t,x)|&\leq C\langle \eta\rangle^{-(1+\delta)}\langle \xi\rangle^{2+2\delta}|Z^{a}u_{\xi}|^2(|u|+|u_{\xi}|+|u_{\eta}|)+C\langle \xi\rangle^{2+2\delta} |Z^{a}u_{\xi}|^2(|u_{\eta}|+|Zu_{\eta}|)  \nonumber\\
&+ C\langle \xi\rangle^{2+2\delta}|Z^{a}u_{\xi}|(|u_{\xi}|+|Zu_{\xi}|)|Z^{a}u_{\eta}|+C\langle \xi\rangle^{2+2\delta}|Z^{a}u_{\eta}|^2(|u_{\xi}|+|Zu_{\xi}|)\nonumber\\
&\leq C|\langle \eta\rangle^{-\frac{1+\delta}{2}}\langle \xi\rangle^{1+\delta}Z^{a}u_{\xi}|^2(|u|+|u_{\xi}|+|u_{\eta}|)\nonumber\\
&+C |\langle \eta\rangle^{-\frac{1+\delta}{2}}\langle \xi\rangle^{1+\delta} Z^{a}u_{\xi}|^2(|\langle \eta\rangle^{1+\delta}u_{\eta}|+|\langle \eta\rangle^{1+\delta}Zu_{\eta}|)  \nonumber\\
&+ C|\langle \eta\rangle^{-\frac{1+\delta}{2}}\langle \xi\rangle^{1+\delta}Z^{a}u_{\xi}||\langle \eta\rangle^{-\frac{1+\delta}{2}}\langle \xi\rangle^{1+\delta}(|u_{\xi}|+|Zu_{\xi}|)||\langle \eta\rangle^{1+\delta}Z^{a}u_{\eta}|\nonumber\\
&+ {\textcolor{blue}{C|\langle \eta\rangle^{-\frac{1+\delta}{2}}\langle \xi\rangle^{1+\delta}Z^{a}u_{\eta}|}}|\langle \eta\rangle^{-\frac{1+\delta}{2}}\langle \xi\rangle^{1+\delta}(|u_{\xi}|+|Zu_{\xi}|)||\langle \eta\rangle^{1+\delta}Z^{a}u_{\eta}|\nonumber\\
&\leq C|\langle \eta\rangle^{-\frac{1+\delta}{2}}\langle \xi\rangle^{1+\delta}Z^{a}u_{\xi}|^2(|u|+|u_{\xi}|+|u_{\eta}|)\nonumber\\
&+C |\langle \eta\rangle^{-\frac{1+\delta}{2}}\langle \xi\rangle^{1+\delta} Z^{a}u_{\xi}|^2(|\langle \eta\rangle^{1+\delta}u_{\eta}|+|\langle \eta\rangle^{1+\delta}Zu_{\eta}|)  \nonumber\\
&+ C\sum_{|b|=2}|\langle \eta\rangle^{-\frac{1+\delta}{2}}\langle \xi\rangle^{1+\delta}Z^{b}u_{\xi}||\langle \eta\rangle^{-\frac{1+\delta}{2}}\langle \xi\rangle^{1+\delta}(|u_{\xi}|+|Zu_{\xi}|)||\langle \eta\rangle^{1+\delta}Z^{a}u_{\eta}|.
\end{align}
Lemma \ref{xuyaouioo99999999} implies
\begin{align}\label{dfgtyuuu766}
|H_a|&\leq C\sum_{|b|\leq 2}|Z^{b}u_{\xi}|\big(|u_{\eta}|+|u_{\eta\eta}|+|u_{\xi\eta}|+|u_{\eta}||u_{\xi\xi}|\big)\nonumber\\
&+C\sum_{|b|\leq 2}|Z^{b}u_{\eta}|\big(|u_{\xi}|+|u_{\xi\xi}|+|u_{\xi\eta}|+|u_{\xi}||u_{\eta\eta}|\big)\nonumber\\
&+C\sum_{|b|+|c|\leq 2}|Z^{b}u_{\xi}||Z^{c}u_{\eta}|.
\end{align}
\par
From \eqref{hj78999}, \eqref{333333}, \eqref{opp9pppp},  we obtain
\begin{align}\label{xddd89uii}
&\sup_{0\leq s\leq t}\sum_{\substack{|a|=2\\a_1\neq 0}}\|\langle \xi\rangle^{1+\delta}Z^{a}u_{\xi}\|_{L_{x}^2}^2+ \sum_{\substack{|a|=2\\a_1\neq 0}}\|\langle\eta\rangle^{-\frac{1+\delta}{2}}\langle \xi\rangle^{1+\delta}Z^{a}u_{\xi}\|_{L^2_{s,x}}^2\nonumber\\
&\leq C E_3(u(0))+C\sup_{0\leq s\leq t}\sum_{\substack{|a|=2\\a_1\neq 0}}\|\widetilde{e}_3(s,\cdot)\|_{L^1(\mathbb{R})}+C\sum_{\substack{|a|=2\\a_1\neq 0}}\int_0^{t}\|q_3(s,\cdot)\|_{L^1(\mathbb{R})}ds\nonumber\\
&+C\sum_{\substack{|a|=2\\a_1\neq 0}}\int_0^{t}\|\langle \xi\rangle^{2+2\delta}Z^{a}u_{\xi}H_a\|_{L_{x}^1(\mathbb{R})}ds.
\end{align}
By \eqref{By3555}, H\"{o}lder inequality, \eqref{224iijuio}, Sobolev embedding $H^1(\mathbb{R})\hookrightarrow L^{\infty}(\mathbb{R})$ and Lemma \ref{keyn9078p},  we have
\begin{align}\label{By35558899}
\|\widetilde{e}_3(t,\cdot)\|_{L^1(\mathbb{R})}
&\leq C\|\langle \xi\rangle^{1+\delta}Z^{a}u_{\xi}\|_{L^2_{x}}^2(\|u\|_{L^{\infty}_{x}}+\|u_{\xi}\|_{L^{\infty}_{x}}+\|u_{\eta}\|_{L^{\infty}_{x}})\nonumber\\
&+C\sum_{|b|=2}\|\langle \xi\rangle^{1+\delta}Z^{b}u_{\xi}\|_{L^2_{x}}\|\langle \xi\rangle^{1+\delta}u_{\xi}\|_{L^{\infty}_{x}}\|Z^{a}u_{\eta}\|_{L^2_{x}}\nonumber\\
&\leq C\big(E_{3}(u(t))+\widetilde{E_{3}}(u(t))\big)\big(E^{1/2}_{1}(u(t))+E^{1/2}_{2}(u(t))\big)\nonumber\\
&\leq C\big(E(u(t))+\widetilde{E_{3}}(u(t))\big)E^{1/2}(u(t))\leq CE^{3/2}(u(t)).
\end{align}
Thanks to \eqref{qguji999980877778}, H\"{o}lder inequality, Lemma \ref{xuyao8DD8899} and Lemma \ref{keyn9078p}, we get
\begin{align}\label{q78907}
&\int_0^{t}\|q_3(s,\cdot)\|_{L^1(\mathbb{R})}ds\nonumber\\
&\leq C\|\langle \eta\rangle^{-\frac{1+\delta}{2}}\langle \xi\rangle^{1+\delta}Z^{a}u_{\xi}\|_{L^2_{s,x}}^2(\|u\|_{L^{\infty}_{s,x}}+\|u_{\xi}\|_{L^{\infty}_{s,x}}+\|u_{\eta}\|_{L^{\infty}_{s,x}})\nonumber\\
&+C \|\langle \eta\rangle^{-\frac{1+\delta}{2}}\langle \xi\rangle^{1+\delta} Z^{a}u_{\xi}\|_{L^{2}_{s,x}}^2(\|\langle \eta\rangle^{1+\delta}u_{\eta}\|_{L^{\infty}_{s,x}}+\|\langle \eta\rangle^{1+\delta}Zu_{\eta}\|_{L^{\infty}_{s,x}})  \nonumber\\
&+ C\sum_{|b|=2}\|\langle \eta\rangle^{-\frac{1+\delta}{2}}\langle \xi\rangle^{1+\delta}Z^{b}u_{\xi}\|_{L^{2}_{s,x}}\|\langle \eta\rangle^{-\frac{1+\delta}{2}}\langle \xi\rangle^{1+\delta}(|u_{\xi}|+|Zu_{\xi}|)\|_{L^{2}_{s}L^{\infty}_{x}}\|\langle \eta\rangle^{1+\delta}Z^{a}u_{\eta}\|_{L^{\infty}_{s}L^{2}_{x}}\nonumber\\
&\leq C\big(\sup_{0\leq s\leq t}{E^{1/2}}(u(s))+\sup_{0\leq s\leq t}\widetilde{E_3}^{1/2}(u(s))\big)\big({\mathcal{E}}(u(t))+\widetilde{\mathcal{E}_3}(u(t))\big)\nonumber\\
&\leq C\sup_{0\leq s\leq t}{E^{1/2}}(u(s)){\mathcal{E}}(u(t)).
\end{align}
H\"{o}lder inequality implies
\begin{align}\label{xuyaoo89676hfftu8990}
&\int_0^{t}\|\langle \xi\rangle^{2+2\delta}Z^{a}u_{\xi}H_a\|_{L_{x}^1(\mathbb{R})}ds\nonumber\\
&\leq C\|\langle \eta\rangle^{-\frac{1+\delta}{2}}\langle \xi\rangle^{1+\delta}
Z^{a}u_{\xi}\|_{L^2_{s,x}}\|\langle \eta\rangle^{\frac{1+\delta}{2}}\langle \xi\rangle^{1+\delta}
H_a\|_{L^2_{s,x}}\nonumber\\
&\leq \mathcal {E}_3^{1/2}(u(t))\|\langle \eta\rangle^{\frac{1+\delta}{2}}\langle \xi\rangle^{1+\delta}
H_a\|_{L^2_{s,x}}.
\end{align}
Via \eqref{dfgtyuuu766}, Lemma \ref{xuyao8DD8899} and Lemma \ref{keyn9078p}, we can see
\begin{align}\label{xuyao8977}
&\|\langle \eta\rangle^{\frac{1+\delta}{2}}\langle \xi\rangle^{1+\delta}
H_a\|_{L^2_{s,x}}\nonumber\\
&\leq C\sum_{|b|\leq 2}\|\langle \eta\rangle^{-\frac{1+\delta}{2}}\langle \xi\rangle^{1+\delta}Z^{b}u_{\xi}\|_{_{L^2_{s,x}}}\|\langle \eta\rangle^{{1+\delta}}\big(|u_{\eta}|+|u_{\eta\eta}|+|u_{\xi\eta}|+|u_{\eta}||u_{\xi\xi}|\big)\|_{{L^{\infty}_{s,x}}}\nonumber\\
&+C\sum_{|b|\leq 2}\|\langle \eta\rangle^{{1+\delta}}Z^{b}u_{\eta}\|_{L^{\infty}_{s}L^2_{x}} \|\langle \eta\rangle^{-\frac{1+\delta}{2}}\langle \xi\rangle^{1+\delta}\big(|u_{\xi}|+|u_{\xi\xi}|+|u_{\xi\eta}|+|u_{\xi}||u_{\eta\eta}|\big)\|_{L^{2}_{s}L^{\infty}_{x}}\nonumber\\
&+C\sum_{\substack{|b|+|c|\leq 2\\|b|\leq 1}}\|\langle \eta\rangle^{-\frac{1+\delta}{2}}\langle \xi\rangle^{1+\delta}Z^{b}u_{\xi}\|_{L^{2}_{s}L^{\infty}_{x}}\|\langle \eta\rangle^{{1+\delta}}Z^{c}u_{\eta}\|_{L^{\infty}_{s}L^2_{x}}\nonumber\\
&+C\sum_{\substack{|b|+|c|\leq 2\\|c|\leq 1}}\|\langle \eta\rangle^{-\frac{1+\delta}{2}}\langle \xi\rangle^{1+\delta}Z^{b}u_{\xi}\|_{_{L^2_{s,x}}}\|\langle \eta\rangle^{{1+\delta}}Z^{c}u_{\eta}\|_{{L^{\infty}_{s,x}}}\nonumber\\
&\leq C\big(\sup_{0\leq s\leq t}{E^{1/2}}(u(s))+\sup_{0\leq s\leq t}\widetilde{E_3}^{1/2}(u(s))\big)\big({\mathcal{E}^{1/2}}(u(t))+\widetilde{\mathcal{E}_3}^{1/2}(u(t))\big)\nonumber\\
&\leq C\sup_{0\leq s\leq t}{E^{1/2}}(u(s)){\mathcal{E}}^{1/2}(u(t)).
\end{align}
It follows from \eqref{xuyaoo89676hfftu8990} and \eqref{xuyao8977} that
\begin{align}\label{dddff6899}
\int_0^{t}\|\langle \xi\rangle^{2+2\delta}Z^{a}u_{\xi}H_a\|_{L_{x}^1(\mathbb{R})}ds\leq C\sup_{0\leq s\leq t}{E^{1/2}}(u(s)){\mathcal{E}}(u(t)).
\end{align}
We conclude from \eqref{xddd89uii}, \eqref{By35558899}, \eqref{q78907} and \eqref{dddff6899} that
\begin{align}\label{xddd89ui9dddddd988i}
&\sup_{0\leq s\leq t}\sum_{\substack{|a|=2\\a_1\neq 0}}\|\langle \xi\rangle^{1+\delta}Z^{a}u_{\xi}\|_{L_{x}^2}^2+ \sum_{\substack{|a|=2\\a_1\neq 0}}\|\langle\eta\rangle^{-\frac{1+\delta}{2}}\langle \xi\rangle^{1+\delta}Z^{a}u_{\xi}\|_{L^2_{s,x}}^2\nonumber\\
&\leq C E_3(u(0))+C\sup_{0\leq s\leq t}E^{3/2}(u(s))+C\sup_{0\leq s\leq t}{E^{1/2}}(u(s)){\mathcal{E}}(u(t)).
\end{align}
\par
Noting the duality between $\xi$ and $\eta$, by a similar way, we can also get
\begin{align}\label{xddddd89ui9988dddddi}
&\sup_{0\leq s\leq t}\sum_{\substack{|a|=2\\a_2\neq 0}}\|\langle \eta\rangle^{1+\delta}Z^{a}u_{\eta}\|_{L_{x}^2}^2+ \sum_{\substack{|a|=2\\a_2\neq 0}}\|\langle\xi\rangle^{-\frac{1+\delta}{2}}\langle \eta\rangle^{1+\delta}Z^{a}u_{\eta}\|_{L^2_{s,x}}^2\nonumber\\
&\leq C E_3(u(0))+C\sup_{0\leq s\leq t}E^{3/2}(u(s))+C\sup_{0\leq s\leq t}{E^{1/2}}(u(s)){\mathcal{E}}(u(t)).
\end{align}
\par
Acordingly, we have obtained
\begin{align}\label{xddddd89uiddd9988dddddi}
&\sup_{0\leq s\leq t}E_3(u(s))+ \mathcal{E}_3(u(t))\nonumber\\
&\leq C E_3(u(0))+C\sup_{0\leq s\leq t}E^{3/2}(u(s))+C\sup_{0\leq s\leq t}{E^{1/2}}(u(s)){\mathcal{E}}(u(t)).
\end{align}

\subsection{Global existence}
Noting \eqref{dijie3455888}, \eqref{dijie3455} and \eqref{xddddd89uiddd9988dddddi}, we have
\begin{align}\label{ZUSHI}
&\sup_{0\leq s\leq t}E(u(s))+ \mathcal{E}(u(t))\nonumber\\
&\leq C E(u(0))+C\sup_{0\leq s\leq t}E^{3/2}(u(s))+C\sup_{0\leq s\leq t}{E^{1/2}}(u(s)){\mathcal{E}}(u(t)).
\end{align}
Under the assumption \eqref{dddddd}, we have
\begin{align}
\sup_{0\leq s\leq t}E(u(s))+\mathcal {E}(u(t))\leq C_1\varepsilon^2+8C_1A^3\varepsilon^3.
\end{align}
Assume that
\begin{align}
E(u(0))\leq \widetilde{C_1}\varepsilon^2.
\end{align}
Taking $A^2=4\max\{C_1,\widetilde{C_1}\}$ and $\varepsilon_0$ so small that
\begin{align}
16C_1A\varepsilon_0\leq 1,
\end{align}
for any $\varepsilon$ with $0<\varepsilon\leq \varepsilon_0$, we have
\begin{align}\label{351}
\sup_{0\leq s\leq t}E(u(s))+\mathcal {E}(u(t))\leq A^2\varepsilon^2.
\end{align}\par
This completes the proof of global existence part of Theorem \ref{mainthm}.
\subsection{Asymptotic behavior}
In view of Lemma \ref{xuyaoodddd897}, in order to prove that the global solution $u$ to the system \eqref{quasiwave} is asymptotically free in the energy sense, we only need to show
\begin{align}\label{cfgtyyyyyyy}
\int_0^{+\infty}\|A_1u_{\xi\eta}\|_{L^2_{x}(\mathbb{R})}
+\|A_2u_{\xi\xi}\|_{L^2_{x}(\mathbb{R})}+\|A_3u_{\eta\eta}\|_{L^2_{x}(\mathbb{R})}+\|F\|_{L^2_{x}(\mathbb{R})} dt<+\infty.
\end{align}
\par
First, by \eqref{order1}, \eqref{order22}, \eqref{order33}, Lemma \ref{xuyao8DD8899}, Lemma \ref{keyn9078p} and \eqref{351}, we have pointwise estimates
\begin{align}\label{xyaio890uhu888}
& |A_1u_{\xi\eta}|
+|A_2u_{\xi\xi}|+|A_3u_{\eta\eta}|+|F|\nonumber\\
&\leq C|u||u_{\xi\eta}|+C|u_{\xi}|(|u_{\eta}|+|u_{\xi\eta}|+|u_{\eta\eta}|)+C|u_{\eta}|(|u_{\eta\xi}|+|u_{\xi\xi}|)\nonumber\\
&\leq CE^{1/2}(u(t))|u_{\xi\eta}|+C\langle \xi\rangle^{-(1+\delta)}\langle \eta\rangle^{-(1+\delta)}E^{1/2}(u(t))\big(E^{1/2}(u(t))+\widetilde{E_3}^{1/2}(u(t))\big)\nonumber\\
&\leq CE^{1/2}(u(t))|u_{\xi\eta}|+C\langle \xi\rangle^{-(1+\delta)}\langle \eta\rangle^{-(1+\delta)}E(u(t))\nonumber\\
&\leq CA\varepsilon|u_{\xi\eta}|+CA^2\varepsilon^2\langle \xi\rangle^{-(1+\delta)}\langle \eta\rangle^{-(1+\delta)}.
\end{align}
In view of the system \eqref{quasiwave}, we also have
 \begin{align}
&|u_{\xi\eta}|\leq |A_1u_{\xi\eta}|
+|A_2u_{\xi\xi}|+|A_3u_{\eta\eta}|+|F|\leq  CA\varepsilon|u_{\xi\eta}|+CA^2\varepsilon^2\langle \xi\rangle^{-(1+\delta)}\langle \eta\rangle^{-(1+\delta)}.
\end{align}
If $\varepsilon$ is sufficiently small, it holds that
\begin{align}\label{eqrefvvffff}
&|u_{\xi\eta}|\leq CA^2\varepsilon^2\langle \xi\rangle^{-(1+\delta)}\langle \eta\rangle^{-(1+\delta)}.
\end{align}
It follows from \eqref{xyaio890uhu888} and \eqref{eqrefvvffff} that
\begin{align}\label{xyaio890uhu}
& |A_1u_{\xi\eta}|
+|A_2u_{\xi\xi}|+|A_3u_{\eta\eta}|+|F|\leq CA^2\varepsilon^2\langle \xi\rangle^{-(1+\delta)}\langle \eta\rangle^{-(1+\delta)}.
\end{align}
\par
Noting that
\begin{align}
&\langle \xi\rangle^{-(1+\delta)}\langle \eta\rangle^{-(1+\delta)}\leq C\langle t+x\rangle^{-(1+\delta)}\langle t-x\rangle^{-(1+\delta)}\nonumber\\
&\leq C\langle t+|x|\rangle^{-(1+\delta)}\langle t-|x|\rangle^{-(1+\delta)}\leq C\langle t\rangle^{-(1+\delta)}\langle t-|x|\rangle^{-(1+\delta)},
\end{align}
we have
\begin{align}\label{yiweio09}
\|\langle \xi\rangle^{-(1+\delta)}\langle \eta\rangle^{-(1+\delta)}\|_{L^2_{x}(\mathbb{R})}\leq C\|\langle t-|x|\rangle^{-(1+\delta)}\|_{L^2_{x}(\mathbb{R})}\langle t\rangle^{-(1+\delta)}\leq C\langle t\rangle^{-(1+\delta)}.
\end{align}
The combination of \eqref{xyaio890uhu} and \eqref{yiweio09} gives
\begin{align}\label{xyaio89ffff0uhu}
& \|A_1u_{\xi\eta}\|_{L^2_{x}(\mathbb{R})}
+\|A_2u_{\xi\xi}\|_{L^2_{x}(\mathbb{R})}+\|A_3u_{\eta\eta}\|_{L^2_{x}(\mathbb{R})}+\|F\|_{L^2_{x}(\mathbb{R})}\leq CA^2\varepsilon^2\langle t\rangle^{-(1+\delta)}.
\end{align}
\eqref{cfgtyyyyyyy} is just a consequence of \eqref{xyaio89ffff0uhu}. \par
This completes the proof of asymptotic behavior part of Theorem \ref{mainthm}.

\section*{Acknowledgements}
The author would like to express his sincere gratitude to Prof. Arick Shao, Prof. Willie Wai-Yeung Wong and Prof. Shiwu Yang for their helpful comments on the topic in this paper.\par
The author was supported by National Natural Science Foundation of China
(No.11801068) and the Fundamental Research Funds for the Central Universities.

\end{document}